\documentclass[12pt]{article}
\usepackage[paperwidth=230mm, paperheight=310mm]{geometry}
\usepackage{graphicx}
\usepackage{amssymb}
\usepackage{amsmath}
\usepackage{hyperref}
\usepackage{enumerate}
\usepackage{amsthm}
\usepackage{amsmath}
\usepackage{float}
\usepackage{color}   
\usepackage{lineno}

\makeatletter
\renewenvironment{proof}[1][\proofname]{%
   \par\pushQED{\qed}\normalfont%
   \topsep6\p@\@plus6\p@\relax
   \trivlist\item[\hskip\labelsep\bfseries#1\@addpunct{.}]%
   \ignorespaces
}{%
   \popQED\endtrivlist\@endpefalse
}
\makeatother

\newtheorem{theorem}{Theorem}
\newtheorem{proposition}[theorem]{Proposition}
 \numberwithin{theorem}{section}
 \newtheorem{corollary}[theorem]{Corollary}
\newtheorem{lemma}[theorem]{Lemma}

\newtheorem{remark}[theorem]{Remark}

\numberwithin{equation}{section}

\renewcommand{\P}{\mathbb{P}}
\newcommand{\E}{\mathbb{E}}
\newcommand{\R}{\mathbb{R}}
\newcommand{\mR}{\mathcal{R}}

\newcommand{\pmR}{\partial\mathcal{R}}

\newcommand{\N}{\mathbb{N}}

\newcommand{\cA}{\mathcal A}

\newcommand{\cN}{\mathcal{N}}

\newcommand{\cL}{\mathcal{L}}

\newcommand{\cF}{\mathcal F}

\newcommand{\cO}{\mathcal O}
\newcommand{\cE}{\mathcal E}

\newcommand{\cW}{\mathcal W}
\newcommand{\eps}{\varepsilon}
\newcommand{\veps}{\varepsilon}
 \newcommand{\nn}{\nonumber}
 \newcommand{\no}{\noindent}

\newcommand{\Extra}[1]{{\color{blue}#1}}
\renewcommand{\Extra}[1]{}

\begin{document}
\author{
Jieliang Hong\footnote{Department of Mathematics, University of British Columbia, Canada, E-mail: {\tt jlhong@math.ubc.ca} }
}
\title{Exit measure, local time and a boundary local time of super-Brownian motion}
\date{\today}
\maketitle
\begin{abstract}
    We use a renormalization of the total mass of the exit measure from the complement of a small ball centered at $x\in \R^d$ for $d\leq 3$ to give a new construction of the total local time $L^x$ of super-Brownian motion at $x$. In \cite{Hong20} a more singular renormalization of the total mass of the exit measure concentrating on $x$, where the exit measure is positive but unusually small, is used to build a boundary local time supported on the topological boundary of the range of super-Brownian motion. Our exit measure construction of $L^x$ motivates this renormalization. We give an important step of this construction here by establishing the convergence of the associated mean measure to an explicit limit; this will be used in the construction of the boundary local time in \cite{Hong20}. Both our results rely on the behaviour of solutions to the associated semilnear elliptic equation with singular initial data and on Le Gall's special Markov property for exit measures.
\end{abstract}

\section{Introduction and main results}

Let $M_F=M_F(\R^d)$ be the space of finite measures on $(\R^d,\mathfrak{B}(\R^d))$ equipped with the topology of weak convergence of measures, and write $\mu(\phi)=\int \phi(x) \mu(dx)$ for $\mu \in M_F$. A super-Brownian motion (SBM) $(X_t, t\geq 0)$ starting at $\mu \in M_F$ is a continuous $M_F$-valued strong Markov process defined on some filtered probability space $(\Omega,\cF,\cF_t,P)$ with $X_0=\mu$ a.s., which arises as the unique solution to the following $\mathit{martingale\ problem}$ (see \cite{Per02}, II.5): For any $ \phi \in C_b^2(\R^d)$, 
\begin{equation} 
X_t(\phi)=X_0(\phi)+M_t(\phi)+\int_0^t X_s(\frac{\Delta}{2}\phi) ds,
\end{equation} where $M_t(\phi)$ is a continuous $\cF_t$-martingale such that $M_0(\phi)=0$ and  \[[M(\phi)]_t=\int_0^t X_s(\phi^2) ds.\]
Here $C_b^2(\R^d)$ is the space of bounded functions which are twice continuously differentiable. The above martingale problem uniquely characterizes the law $\P_{X_0}$ of super-Brownian motion $X$, starting from $X_0 \in M_F$, on $C([0,\infty),M_F)$, the space of continuous functions from $[0,\infty)$ to $M_F$ furnished with the compact-open topology.

%
%
We know that the extinction time of $X$ is a.s. finite (see, e.g., Chp II.5 in \cite{Per02}). The total occupation time measure of $X$ is the (a.s. finite) measure defined as
 \begin{eqnarray*}
I(A)=\int_0^\infty X_s(A)ds.
\end{eqnarray*}
Let $S(\mu)=\text{Supp}(\mu)$ denote the closed support of a measure $\mu$. We 
define the range, $\mR$, of $X$ to be 
$\mR=\textnormal{Supp}(I).$
In dimensions $d\leq 3$, the occupation measure $I$ has a density, $L^x$, which is called (total) local time of $X$, that is,
$$I(f)=\int_0^\infty X_s(f)\,ds=\int_{\R^d} f(x)L^x\,dx\text{ for all  non-negative measurable }f.$$
Moreover, $x\mapsto L^x$ is lower semicontinuous, is continuous on $S(X_0)^c$, and  for $d=1$ is globally continuous (see Theorems~2 and 3 of \cite{Sug89}). Thus one can see that in dimensions $d\leq 3$,
\begin{equation}\label{er1.0}
\mR=\overline{\{x:L^x>0\}},
\end{equation}
and $\mR$ is a closed set of positive Lebesgue measure. In dimensions $d\ge 4$, $\mR$ is a Lebesgue null
set of Hausdorff dimension $4$ for SBM starting from $\delta_0$ (see Theorem~1.4 of \cite{DIP89}), which explains our restriction to $d\le 3$ in this work. The Laplace transform of $L^x$ derived in Lemma 2.2 of \cite{MP17} is given by
\begin{equation}\label{ev1.0}
\E_{X_0}\Big(\exp(-\lambda L^x)\Big)=\exp\left(-\int_{\R^d} V^\lambda(x-y)X_0(dy)\right),
\end{equation}
where $V^\lambda$ is the unique solution (see Section~2 of \cite{MP17} and the references given there) to 
\begin{equation}\label{ev1.1}
\frac{\Delta V^\lambda}{2}=\frac{(V^\lambda)^2}{2}-\lambda\delta_0,\ \ V^\lambda>0\text{ on }\R^d.
\end{equation}
Let $\lambda\uparrow\infty$ in~\eqref{ev1.0} and \eqref{ev1.1} to see that $V^\lambda(x)\uparrow V^\infty(x)$ where
\begin{equation}\label{ev1.2}
\P_{X_0}(L^x=0)=\exp\Big(-\int V^\infty(x-y) X_0(dy)\Big).
\end{equation}
It is explicitly known that (see, e.g., (2.17) in \cite{MP17}):
\begin{equation}\label{ev1.3}
V^\infty(x)=\frac{2(4-d)}{|x|^2}:=\frac{\lambda_d}{|x|^2},
\end{equation}
and in particular $V^\infty$ solves
\begin{equation}\label{ev1.4}
\frac{\Delta V^\infty}{2}=\frac{(V^\infty)^2}{2}\text{ for }x\neq 0.
\end{equation}
%
%

%

It is also natural to consider SBM under the canonical measure $\N_{x_0}$. Recall from Section II.7 in \cite{Per02} that $\N_{x_0}$ is a $\sigma$-finite measure on $C([0,\infty),M_F)$ such that if we let
$\Xi=\sum_{i\in I}\delta_{\nu^i}$ be a Poisson point process on $C([0,\infty),M_F)$ with intensity $\N_{X_0}(d\nu)=\int \N_{x}(d\nu) X_0(dx)$, then
\begin{equation}\label{ev2.1}
X_t=\sum_{i\in I}\nu_t^i=\int \nu_t\ \Xi(d\nu),\ t>0, 
\end{equation}
has the law, $\P_{X_0}$, of a super-Brownian motion $X$ starting from $X_0$. In this way, $\N_{x_0}$ describes the contribution of a cluster from a single ancestor at $x_0$ and the super-Brownian motion is then obtained by a Poisson superposition of such clusters. 
We refer the readers to Theorem II.7.3(c) in \cite{Per02} for more details. The existence of the local time $L^x$ under $\N_{x_0}$ will follow from this decomposition and the existence under $\P_{\delta_{x_0}}$. Therefore the local time $L^x$ under $\P_{X_0}$ may be decomposed as 
\begin{align}\label{e1.21}
L^x=\sum_{i\in I}L^x(\nu^i)=\int L^x(\nu)\Xi(d\nu).
\end{align}
The global continuity of local times $L^x$ under $\N_{x_0}$ is given in Theorem 1.2 of \cite{Hong18}. Apply \eqref{e1.21} in \eqref{ev1.0} and \eqref{ev1.2} to get $\N_{x_0}(1-e^{-\lambda L^x})=V^\lambda(x-x_0)$ and $\N_{x_0}(L^x>0)=V^\infty(x-x_0)$ for all $x\neq x_0$.

Turning to the exit measure, we let $d(x,K)=\inf\{|x-y|: y\in K\}$ and for $K_1,K_2$ non-empty, we set 
\[d(K_1,K_2)=\inf\{|x-y|: x\in K_1, y\in K_2\}.\]
Define
\begin{align} \label{Gdef}
\cO_{X_0}\equiv &\{ \text{open sets }  D \text{ satisfying } d(D^c,S(X_0))>0\text{ and a Brownian }\\ \nn&\text{  
 path 
starting from any $x\in\partial D$ will exit $D$ immediately}\}.
\end{align}
In what follows we always assume that $G\in \cO_{X_0}$. The exit measure of $X$ from an open set $G$, under $\P_{X_0}$ or $\N_{X_0}$, is a random finite measure supported on $\partial G$ and is denoted by $X_G$ (see Chp. V of \cite{Leg99} for this and the construction of the exit measure). Intuitively $X_{G}$ corresponds to the mass started at $X_0$ which is stopped at the instant it leaves $G$. We note \cite{Leg99} also suffices as a reference for the properties of $X_G$ described below.
The Laplace functional of $X_G$ is given by 
\begin{align}\label{e7.1}
\E_{X_0}\Big(e^{-X_G(g)}\Big)=\exp\Bigl(-\N_{X_0}\big(1-\exp(-X_G(g))\big)\Bigr)= \exp\Bigl(-X_0(U^g)\Bigr),
\end{align}
where $g:\partial G\to[0,\infty)$ is continuous and $U^g\ge 0$ is the unique continuous function on $\overline G$ which is $C^2$ on $G$ and solves
\begin{equation}\label{e7.2}
\Delta U^g=(U^g)^2\text{ on }G,\quad U^g=g\text{ on }\partial G.
\end{equation}
Let $B(x_0, \eps)=B_\eps(x_0)=\{x:|x-x_0|<\varepsilon\}$ and set $B_\eps=B(\eps)=B_\eps(0)$. Similarly we define
\begin{equation}\label{e0.0.1}
G_\varepsilon^{x_0}=G_\varepsilon(x_0)=\{x:|x-x_0|>\varepsilon\} \text{ and set } G_\veps=G_\veps(0).
\end{equation}
For $\veps>0$ and $\lambda\geq 0$, we let $U^{\lambda,\veps}$ denote the unique continuous function on $\{|x|\ge \veps\}$ such that (cf. \eqref{e7.2})
\begin{equation}\label{ev5.2}
\Delta U^{\lambda,\veps}=(U^{\lambda,\veps})^2\ \ \text{for }|x|>\veps,\ \ \text{ and }\  U^{\lambda,\veps}(x)=\lambda\ \ \text{for }|x|=\veps.
\end{equation}
Uniqueness of solutions implies the scaling property
\begin{equation}\label{ev5.3}
U^{\lambda,\veps}(x)=\veps^{-2} U^{\lambda \veps^2,1}(x/\veps)\quad\text{for all }|x|\ge\veps,
\end{equation}
and also shows $U^{\lambda, \veps}$ is radially symmetric, thus allowing us to write $U^{\lambda,\veps}(|x|)$ for the value at $x\in\R^d$. 
 By \eqref{e7.1}, we have for any initial condition $X_0\in M_F$ satisfying $S(X_0)\subset G_\veps$,
\begin{equation}\label{ev5.4}
\E_{X_0}\Big(e^{-\lambda X_{G_\veps}(1)}\Big)=\exp\Bigl(-\N_{X_0}\big(1-\exp(-\lambda X_{G_\veps}(1))\big)\Bigr)=\exp(-X_0(U^{\lambda, \veps})).
\end{equation}
Let $\lambda\uparrow\infty$ in the above to see that $U^{\lambda,\veps}\uparrow U^{\infty,\veps}$ on $G_\veps$ and 
\begin{equation}\label{ev5.5}
\P_{X_0}(X_{G_\veps}(1)=0)=\exp(-X_0(U^{\infty,\veps})).
\end{equation}
Proposition 9(iii) of \cite{Leg99} readily implies (see also (3.5) and (3.6) of \cite{MP17})
\begin{align}\label{ev5.6}
U^{\infty,\veps}\text{ is }C^2\text{ and }&\Delta U^{\infty,\veps}=(U^{\infty,\veps})^2\text{ on }G_\veps,\\
 \nonumber &\lim_{|x|\to\veps,|x|>\veps}U^{\infty,\veps}(x)=+\infty,\ \lim_{|x|\to\infty}U^{\infty,\veps}(x)=0.
\end{align}

 It has been shown in Proposition 6.2(b) of \cite{HMP18} that for any $x\in S(X_0)^c$, under $\N_{X_0}$ or $\P_{X_0}$, the family $\{X_{G_{r_0-r}^x}(1), 0\leq r<r_0\}$ with $r_0=d(x,S(X_0))/2$ has a cadlag version which is a supermartingale if $d=3$, and is a martingale if $d=2$. Throughout the rest of the paper, we will always work with this cadlag version. For any $\eps>0$, we let 
 \begin{align}\label{e3.1}
 \psi_0(\eps)=
\begin{cases}
\frac{1}{\pi}\log^{+}(1/\eps), &\text{ in } d=2,\\
\frac{1}{2\pi}\frac{1}{\eps}, &\text{ in } d=3.
\end{cases}
 \end{align}
The following result gives a new construction of the local time $L^x$ in terms of the local asymptotic behavior of the exit measures at $x$, whose proof will be given in Section \ref{s3}.

 \begin{theorem}\label{t0}
 Let $d=2$ or $d=3$. For any $x\in S(X_0)^c$, we have
\begin{equation}
X_{G_\eps^x}(1)\psi_0(\eps) \to L^x \text{ in measure under }   \N_{X_0} \text{ or } \P_{X_0} \text{ as } \eps \downarrow 0,
\end{equation}
where $\psi_0$ is as in \eqref{e3.1}.
 Moreover, in $d=3$ the convergence holds  $\N_{X_0}$-a.e. or $\P_{X_0}$-a.s.
\end{theorem}
\begin{remark}
In $d=3$, the family $\cA:=\{X_{G_{r_0-r}^x}(1)\psi_0(r_0-r), 0\leq r<r_0\}$ with \break $r_0=d(x,S(X_0))/2$ is indeed a martingale (see the proof of the above theorem in Section \ref{s3}). In $d=2$, we already have the family $\{X_{G_{r_0-r}^x}(1), 0\leq r<r_0\}$ is a martingale and so one can easily conclude that $\cA$ is a submartingale in $d=2$.
\end{remark}
A second result on the exit measure stems from a recent work on the construction of a boundary local time measure that is supported on the topological boundary of the range, $\pmR$, of SBM (see \cite{Hong20}). Let 
\begin{equation}\label{ep1.1}
p=p(d)=\begin{cases}
3 &\text{ if }d=1\\
2\sqrt 2 &\text{ if }d=2\\
\frac{1+\sqrt {17}}{2} &\text{ if }d=3,
\end{cases}
\text{ and } \alpha=\frac{p-2}{4-d}=\begin{cases}
1/3 &\text{ if }d=1\\
\sqrt 2-1 &\text{ if }d=2\\
\frac{\sqrt {17}-3}{2} &\text{ if }d=3.
\end{cases}
\end{equation}
For any $\lambda>0$, under $\N_{0}$ or $\P_{\delta_0}$ , in \cite{Hong20} we define a measure ${\cL}^\lambda$ by 
\begin{align}\label{e7.3}
d{\cL}^\lambda(x)=
 \lambda^{1+\alpha} L^x e^{-\lambda L^x} dx.
\end{align}
Theorem 1.3 of \cite{Hong20} gives that in $d=2$ or $3$, there exists a random measure $\cL\in M_F$ such that $\cL^\lambda$ converges in measure to $\cL$ under $\P_{\delta_0}$ or $\N_{0}$ as $\lambda \to \infty$. To prove that the support of $\cL$ is exactly $\pmR$, for any $\kappa, \eps>0$, under $\P_{\delta_0}$ or $\N_{0}$ we further define a second measure $\widetilde{\cL}(\kappa)^\eps$ in \cite{Hong20} by
\begin{equation}\label{edef}
d\widetilde{\cL}(\kappa)^\eps(x)=
\frac{X_{G_\eps^x}(1)}{\eps^p} \exp(-\kappa \frac{X_{G_\eps^x}(1)}{\eps^{2}} ) 1(X_{G_{\eps/2}^x}=0) 1(|x|>\eps)dx.
 \end{equation}
The indicator function $1(|x|>\eps)$ is to ensure that $X_{G_\eps^x}$ is well defined. The extra indicator $1(X_{G_{\eps/2}^x}=0)$ is to ensure that the limiting measure will be supported on $\pmR$ rather than the bigger set $F$, the boundary of the zero set of $L^x$ (see (1.8) of \cite{HMP18}). 

We write $f(t)\sim g(t)$ as $t\downarrow 0$ iff $f(t)/g(t)$ is bounded below and above by constants $c, c'>0$ for small positive $t$. One can deduce from Theorem \ref{t0} that $\widetilde{\cL}(\kappa)^\eps$ is closely related to ${\cL}^\lambda$: for example in $d=3$, we have $\psi_0(\eps)=1/(2\pi \eps)$, and then by Theorem \ref{t0} we have $X_{G_\eps^x}(1) \sim \eps L^x$ as $\eps \downarrow 0$. Hence if $\lambda=\kappa \eps^{-1}$,
\begin{align}
\frac{X_{G_\eps^x}(1)}{\eps^p} \exp(-\kappa\frac{X_{G_\eps^x}(1)}{\eps^{2}})\sim \eps^{1-p} L^x e^{-\kappa \eps^{-1} L^x}\sim \lambda^{1+\alpha} L^x e^{-\lambda L^x} \text{ as } \eps \downarrow 0,
\end{align}
where in the last approximation we have used the fact that $\alpha=p-2$ in $d=3$.  
 In fact Theorem 1.13 of \cite{Hong20} readily implies that in $d=2$ or $3$, there is some constant $c(\kappa)>0$ so that $\widetilde{\cL}(\kappa)^\eps$ converges in measure to $c(\kappa)\cL$ under $\P_{\delta_0}$ or $\N_{0}$ as $\eps\downarrow 0$. In this paper we prove the convergence of the mean measure of $\widetilde{\cL}(\kappa)^\eps$ as $\eps\downarrow 0$, which will play a role in the convergence of $\widetilde{\cL}(\kappa)^\eps$ to $c(\kappa)\cL$  in \cite{Hong20}. We briefly include $d=1$ in our results.\\

\no ${\bf Convention\ on\ Functions\ and\ Constants}$ Constants whose value is unimportant and may change from line to line are denoted $C, c, c_1,c_2,\dots$, while constants whose values will be referred to later and appear initially in say, Lemma~i.j are denoted $c_{i.j},$ or $ \underline c_{i.j}$ or $C_{i.j}$. 


\begin{theorem}\label{t1}
Let $d\leq 3$ and $X_0\in M_F(\R^d)$. For any $\kappa>0$, there is some constant \mbox{$C_{\ref{t1}}(\kappa)>0$} such that for all $x\neq 0$,
 \begin{align}\label{ce9.3.0}
&\lim_{\eps \downarrow 0} \N_0\Big(\frac{X_{G_\eps^x}(1)}{\eps^p} \exp(-\kappa \frac{X_{G_\eps^x}(1)}{\eps^{2}})1(X_{G_{\eps/2}^x}=0) \Big)=C_{\ref{t1}}(\kappa) |x|^{-p}, 
\end{align}
and 
 \begin{align}\label{ce9.1.0}
& \N_0\Big(\frac{X_{G_\eps^x}(1)}{\eps^p} \exp(-\kappa \frac{X_{G_\eps^x}(1)}{\eps^{2}})1(X_{G_{\eps/2}^x}=0) \Big)\leq |x|^{-p}, \ \forall 0<\eps<|x|.
\end{align}
For any $x\in S(X_0)^c$, we have
 \begin{align}\label{cea9.1.0}
\lim_{\eps \downarrow 0} \E_{X_0}&\Big(\frac{X_{G_\eps^x}(1)}{\eps^p} \exp(-\kappa \frac{X_{G_\eps^x}(1)}{\eps^{2}})1(X_{G_{\eps/2}^x}=0) \Big)\nn\\
&= e^{-\int V^\infty(y-x)X_0(dy)}C_{\ref{t1}}(\kappa) \int |y-x|^{-p} X_0(dy).
\end{align}
 and
 \begin{align}\label{cea9.3.0}
 &\E_{X_0}\Big(\frac{X_{G_\eps^x}(1)}{\eps^p} \exp(-\kappa \frac{X_{G_\eps^x}(1)}{\eps^{2}})1(X_{G_{\eps/2}^x}=0) \Big)\nn\\
&\leq  \int |y-x|^{-p} X_0(dy),\ \forall  0<\eps<d(x,S(X_0)).
\end{align}
\end{theorem}

The proof of Theorem \ref{t1} will be given in Section \ref{s4}.

\section*{Acknowledgements}
This work was done as part of the author's graduate studies at the University of British Columbia. I would like to thank my supervisor, Professor Edwin Perkins, for suggesting this problem and for the helpful discussions and suggestions throughout this work.

\section{Exit Measures and the Special Markov Property}\label{s2}

We will use Le Gall's Brownian snake construction of a SBM $X$, with initial state \break $X_0\in M_F(\R^d)$. Set $\cW=\cup_{t\ge 0} C([0,t],\R^d)$ with the natural metric (see page 54 of \cite{Leg99}), and let $\zeta(w)=t$ be the lifetime of \mbox{$w\in C([0,t],\R^d)\subset\cW$.} The Brownian snake $W=(W_t,t\ge0)$ is a $\cW$-valued continuous strong Markov process and, abusing notation slightly, let $\N_x$ denote its excursion measure starting from the path at $x\in\R^d$ with lifetime zero.  As usual we let $\hat W(t)=W_t(\zeta(W_t))$ denote the tip of the snake at time $t$, and $\sigma(W)>0$ denote the length of the excursion path. We refer the reader to Ch. IV of \cite{Leg99} for the precise definitions.  The construction of super-Brownian motion, $X=X(W)$ under $\N_x$ or $\P_{X_0}$, may be found in Ch. IV of \cite{Leg99}. The ``law" of $X(W)$ under $\N_x$ is the canonical measure of SBM starting at $x$ described in the last Section (and also denoted by $\N_x$). If $\Xi=\sum_{j\in J}\delta_{W_j}$ is a Poisson point process on $\cW$ with intensity $\N_{X_0}(dW)=\int\N_x(dW)X_0(dx)$, then by Theorem~4 of Ch. IV of \cite{Leg99} (cf. \eqref{ev2.1}), we have
\begin{equation}\label{Xtdec}
X_t(W)=\sum_{j\in J}X_t(W_j)=\int X_t(W)\Xi(dW)\text{ for }t>0
\end{equation}
defines a SBM with initial measure $X_0$. We will refer to this as the standard set-up for $X$ under $\P_{X_0}$. It follows that the total local time $L^x$ under $\P_{X_0}$ may also be decomposed as 
\begin{equation}\label{ae4.1}
L^x=\sum_{j\in J}L^x(W_j)=\int L^x(W)\Xi(dW).
\end{equation}

Recall $\mathcal{R}=\overline{\{x:L^x>0\}}$ is the range of the SBM $X$ under $\P_{X_0}$ or $\N_{X_0}$.    Under $\N_{X_0}$ we have (see (8) on p. 69 of \cite{Leg99})
\begin{equation}\label{ea0.0}
\mR=\{\hat W(s):s\in[0,\sigma]\}.
\end{equation}
Let $G\in\cO_{X_0}$ as in \eqref{Gdef}.   
Then 
\begin{equation}\label{ea1.1}
X_G \text{ is a finite random measure supported on }\partial G \cap \mR\text{ a.s.}
\end{equation}
Under $\N_{X_0}$ this follows from the definition of $X_G$ on p. 77 of \cite{Leg99} and 
the ensuing discussion.
  Although \cite{Leg99} works under $\N_x$ for $x\in G$ the above extends immediately to $\P_{X_0}$ because as in (2.23) of \cite{MP17}, 
  \begin{equation}\label{ea1.2}
X_G=\sum_{j\in J} X_G(W_j)=\int X_G(W)d\Xi(W),
\end{equation}
where $\Xi$ is a Poisson point process on $\cW$ with intensity $\N_{X_0}$.  

Working under $\N_{X_0}$ and following \cite{Leg95}, we define
\begin{align}\label{ae8.1}
S_G(W_u)&=\inf\{t\le \zeta_u: W_u(t)\notin G\}\ \ (\inf\emptyset=\infty),\nn\\
\eta_s^G(W)&=\inf\{t:\int _0^t1(\zeta_u\le S_G(W_u))\,du>s\},\nn\\
\cE_G&=\sigma(W_{\eta_s^G},s\ge 0)\vee\{\N_{X_0}-\text{null sets}\},
\end{align}
where $s\to W_{\eta^G_s}$ is continuous (see p. 401 of \cite{Leg95}).  Write the open set \mbox{$\{u:S_G(W_u)<\zeta_u\}$} as countable union of disjoint open intervals, $\cup_{i\in I}(a_i,b_i)$.
Clearly $S_G(W_u)=S^i_G<\infty$ for all $u\in [a_i,b_i]$ and we may define
\[W^i_s(t)=W_{(a_i+s)\wedge b_i}(S^i_G+t)\text{ for }0\le t\le \zeta_{(a_i+s)\wedge b_i}-S^i_G.\]
Therefore for $i\in I$, $W^i\in C(\R_+,\cW)$ are the excursions of $W$ outside $G$. Proposition 2.3 of \cite{Leg95} implies $X_G$ is $\cE_G$-measurable and Corollary~2.8 of the same reference implies
\begin{equation}\label{SMP1}
\left\{\begin{array}{l}
\text{Conditional on $\cE_G$, the point measure $\sum_{i\in I}\delta_{W^i}$ is a Poisson}\\
\text{point measure with intensity $\N_{X_G}$.}\end{array}\right.
\end{equation}
If $D$ is an open set in $\cO_{X_0}$ 
such that $\bar{G}\subset D$ and $d(D^c,\bar G)>0$, then the definition (and existence) of $X_D(W)$ applies equally well to each $X_D(W^i)$
and it is easy to check that 
\begin{equation}\label{Widecomp}
X_D(W)=\sum_{i\in I}X_D(W^i).
\end{equation}

If $U$ is an open subset of $S(X_0)^c$, then $L_U$, the
restriction of the local time $L^x$ to $U$, is in $C(U)$ which is the set of continuous functions on $U$.
Here are some simple consequences of \eqref{SMP1}. 
\begin{proposition}\label{pv0.2}
Let $G_1, G_2$ be open sets 
 in $\cO_{X_0}$ such that $\overline{G_1}\subset G_2$ and $d(G_2^c,\overline{G_1})>0$.\\
(i) If $\psi_1:C(\overline{G_1}^c)\to [0,\infty)$ is Borel measurable, then
\begin{equation*} 
\N_{X_0}(\psi_1(L_{\overline{G_1}^c})|\cE_{G_1})=\E_{X_{G_1}}(\psi_1(L_{\overline{G_1}^c})).
\end{equation*}
(ii) If $\psi_2:M_F(\R^d)\to [0,\infty)$ is Borel measurable then
\begin{equation*} 
\N_{X_0}(\psi_2(X_{G_2})|\cE_{G_1})=\E_{X_{G_1}}(\psi_2(X_{G_2})).
\end{equation*}
\end{proposition}
\begin{proof}
It follows immediately from Proposition 2.2(a) in \cite{HMP18}. 
\end{proof}

We will need a version of the above under $\P_{X_0}$ as well, which is Proposition 2.3 in \cite{HMP18}.
\begin{proposition}\label{pv0.1}
For 
 $X_0 \in M_F(\R^d)$ and an open set $G$ in $\cO_{X_0}$, let $\Psi$ be a  bounded measurable function on $C({\overline{G}}^c)$ and $\Phi_i$, $i=0,1$ be bounded measurable
functions on $M_F(\R^d)$ and $M_F(\R^d)^n$, respectively. Then\\
(a) $\E_{X_0} (\Phi_0(X_G) \Psi(L))=\E_{X_0}(\Phi_0(X_G)\E_{X_G}(\Psi(L)))$.\\
(b) Let $D_i$ be open sets 
 in $\cO_{X_0}$, such that $d(D_i^c,\bar G)>0$, $\forall 1\leq i\leq n$. Then \[\E_{X_0} \Big(\Phi_0(X_G)  \Phi_1(X_{D_1},\dots, X_{D_n})\Big)=\E_{X_0}\Big(\Phi_0(X_G) \E_{X_G}\Big( \Phi_1(X_{D_1},\dots,X_{D_n})\Big)\Big).\]
\end{proposition}

 \section{Construction of the local time by exit measure}\label{s3}
In this section we will give the proof of Theorem \ref{t0} and we assume throughout this section that $d=2$ or $d=3$. If $V^{\lambda}(x)=V^{\lambda}(|x|)$ is as in \eqref{ev1.1} and $\psi_0$ is as in \eqref{e3.1}, then
Lemma 8 of \cite{BO87} (see more precise results in Remark 1 of the same reference and \cite{Hong18}) shows that
\begin{align}\label{ev2.2}
    \lim_{x\to 0} \frac{V^\lambda(x)}{\psi_0(|x|)}=\lambda.
\end{align}

\begin{proof}[Proof of Theorem \ref{t0}]
We first consider the $\N_{X_0}$ case. Fix any $x\in S(X_0)^c$ and let $\delta=d(x,S(X_0))>0$. For any $\lambda>0$ and $0<\eps<\delta/2$, we have
\begin{align*}
&\N_{X_0}\Big(\Big(\exp(-\lambda X_{G_\eps^x}(1)\psi_0(\eps))-\exp(-\lambda L^x)\Big)^2\Big)\\
=&\N_{X_0}\Big(\exp(-2\lambda X_{G_\eps^x}(1)\psi_0(\eps))+\exp(-2\lambda L^x)-2\exp(-\lambda X_{G_\eps^x}(1)\psi_0(\eps))\exp(-\lambda L^x) \Big)\\
=&\N_{X_0}\Big(\exp(-2\lambda X_{G_\eps^x}(1)\psi_0(\eps))+\E_{X_{G_\eps^x}}\Big(\exp(-2\lambda L^x)\Big)\\
&\quad \quad \quad \quad \quad \quad -2\exp(-\lambda X_{G_\eps^x}(1)\psi_0(\eps))\E_{X_{G_\eps^x}}\Big(\exp(-\lambda L^x)\Big) \Big),
\end{align*}
where we have used Proposition \ref{pv0.2} (i) in the last line.  Apply \eqref{ev1.0} with $X_0=X_{G_\eps^x}$ to see that the above equals
\begin{align*}
&\N_{X_0}\Big(\exp(-2\lambda X_{G_\eps^x}(1)\psi_0(\eps))+\exp(- X_{G_\eps^x}(1) V^{2\lambda}(\eps) )\\
&\quad \quad \quad \quad \quad \quad-2\exp(-\lambda X_{G_\eps^x}(1)\psi_0(\eps))\exp(- X_{G_\eps^x}(1) V^{\lambda}(\eps) )\Big) \\
=&\N_{X_0}\Big(\exp(-2\lambda X_{G_\eps^x}(1)\psi_0(\eps))-\exp(-\lambda X_{G_\eps^x}(1)\psi_0(\eps))\exp(- X_{G_\eps^x}(1) V^{\lambda}(\eps) )\Big)\\
&\quad  \quad +\N_{X_0}\Big(\exp(- X_{G_\eps^x}(1) V^{2\lambda}(\eps) )-\exp(-\lambda X_{G_\eps^x}(1)\psi_0(\eps))\exp(- X_{G_\eps^x}(1) V^{\lambda}(\eps) )\Big) \\
:=&I_1+I_2.
\end{align*}
We first deal with $I_1$.
\begin{align}\label{e4.3.1}
|I_1|\leq &\N_{X_0}\Big(\Big|\exp\Big(-2\lambda X_{G_\eps^x}(1)\psi_0(\eps)\Big)-\exp\Big(-(\lambda+\frac{V^{\lambda}(\eps)}{\psi_0(\eps)}) X_{G_\eps^x}(1)\psi_0(\eps)\Big)\Big|\Big)\nn\\
=& \N_{X_0}\Big(\Big|X_{G_\eps^x}(1)\psi_0(\eps) \exp\Big(-\lambda'(\eps) X_{G_\eps^x}(1)\psi_0(\eps)\Big) \Big(2\lambda-(\lambda+\frac{V^{\lambda}(\eps)}{\psi_0(\eps)})\Big)\Big|\Big)   \nn\\
\leq & \big|\lambda-\frac{V^{\lambda}(\eps)}{\psi_0(\eps)}\big| \cdot \N_{X_0}\Big(X_{G_\eps^x}(1)\psi_0(\eps)\exp\Big(-\lambda X_{G_\eps^x}(1)\psi_0(\eps)\Big) \Big),
\end{align}
where the second line is by the mean value theorem with $\lambda'(\eps)(\omega)$ chosen between $2\lambda$ and $\lambda+V^{\lambda}(\eps)/\psi_0(\eps)$ and the last line follows by $\lambda'(\eps)>\lambda$ for $\eps>0$ small (see \eqref{ev2.2}). 

Recall $\delta=d(x,S(X_0))$ and define $S(X_0)^{>\delta/4} \equiv \{x: d(x,S(X_0))>\delta/4\}$ so that $\overline{B(x,\eps)}\subset S(X_0)^{>\delta/4}$ for $0<\eps<\delta/2$. By the definition of $\mR$ as in \eqref{er1.0} and \eqref{ea1.1}, we can conclude that
\begin{align}\label{er1.1}
\mR\cap S(X_0)^{>\delta/4} = \emptyset \text{ implies } L^x=0  \text{ and } X_{G_\eps^x}(1)=0, \forall 0<\eps<\delta/2. 
\end{align}
Therefore we have \eqref{e4.3.1} becomes
\begin{align*}
|I_1|\leq & \big|\lambda-\frac{V^{\lambda}(\eps)}{\psi_0(\eps)}\big| \cdot \N_{X_0}\Big(X_{G_\eps^x}(1)\psi_0(\eps)\exp\Big(-\lambda X_{G_\eps^x}(1)\psi_0(\eps)\Big) 1(\mR\cap  S(X_0)^{>\delta/4} \neq \emptyset)\Big) \\
\leq &\big|\lambda-\frac{V^{\lambda}(\eps)}{\psi_0(\eps)}\big| \cdot \lambda^{-1} e^{-1} \N_{X_0}(\mR\cap S(X_0)^{>\delta/4} \neq \emptyset) \to 0 \text{ as } \eps \downarrow 0,
\end{align*}
where the second inequality is by $xe^{-\lambda x} \leq \lambda^{-1} e^{-1}, \forall x\geq 0$ and the convergence to $0$ follows from \eqref{ev2.2} and Proposition VI.2 of  \cite{Leg99}. 

Similarly we get $I_2 \to 0$ as $\eps \downarrow 0$. Therefore for any sequence $\eps_{n} \downarrow 0$, we can take a subsequence $\eps_{n_k} \downarrow 0$ so that $\N_{X_0}$-a.e. $\exp(-\lambda X_{G_{\eps_{n_k}}^x}(1)\psi_0( {\eps_{n_k}}))\to \exp(-\lambda L^x)$ as $\eps_{n_k} \downarrow 0$, which implies $\N_{X_0}$-a.e. we have $X_{G_{\eps_{n_k}}^x}(1)\psi_0( {\eps_{n_k}}) \to L^x$ as $\eps_{n_k} \downarrow 0$. By working with the finite measure $\N_{X_0}(\cdot \cap \{\mR \cap S(X_0)^{>\delta/4}\neq \emptyset\})$ and recalling \eqref{er1.1}, we see that the proof of convergence in measure $\N_{X_0}$ is complete. For the $\P_{X_0}$ case, the above calculation works in a similar way and so we omit the details.\\

 In $d=3$, by (6.10) in \cite{HMP18}, for any $x\in S(X_0)^c$ and $0<\eps<d(x,S(X_0))$ we have
 \begin{equation}\label{ev1.2.1}
\E_{X_0}(X_{G_\eps^x}(1))=\N_{X_0}(X_{G_\eps^x}(1))=\int \frac{\eps}{|x-x_0|}dX_0(x_0).
\end{equation}
Therefore for any $0<\eps_2<\eps_1<r_0$ with $r_0=d(x,S(X_0))/2>0$, by Proposition \ref{pv0.2}(ii) we have
 \begin{equation}
 \N_{X_0}\Big(\frac{X_{G_{\eps_2}^x}(1)}{\eps_2}\Big|\cE_{G_{\eps_1}^x}\Big)= \P_{X_{G_{\eps_1}^x}}\Big(\frac{X_{G_{\eps_2}^x}(1)}{\eps_2} \Big)=\frac{X_{G_{\eps_1}^x}(1)}{\eps_1},
\end{equation}
the last by \eqref{ev1.2.1}.
Recall that $\psi_0(\eps)=1/(2\pi \eps)$ and so
 \begin{equation}\label{ev1.2.2}
 \N_{X_0}\Big(X_{G_{\eps_2}^x}(1)\psi_0(\eps_2)\Big|\cE_{G_{\eps_1}^x}\Big)=X_{G_{\eps_1}^x}(1)\psi_0(\eps_1),
\end{equation}
which implies $\{X_{G_{r_0-r}^x}(1)\psi_0(r_0-r), 0\leq r<r_0\}$ is a nonnegative martingale. Recall that we always work with the cadlag version of $X_{G_{r_0-r}^x}(1)$ on $0\leq r<r_0$. Then we may apply martingale convergence theorem to get $\N_{X_0}$-a.e. $\lim_{r\to r_0} X_{G_{r_0-r}^x}(1)\psi_0(r_0-r)$ exists. Since we already have $X_{G_{\eps}^x}(1)\psi_0(\eps)$ converges to $L^x$ in measure $\N_{X_0}$, we conclude that $\N_{X_0}$-a.e. $X_{G_{\eps}^x}(1)\psi_0(\eps) \to L^x$ as $\eps \downarrow 0$. The case for $\P_{X_0}$ follows in a similar way.
\end{proof}

\section{Convergence of the mean measure}\label{s4}
 In this section the proof of Theorem \ref{t1} will be given and we assume $d\leq 3$. First we give some preliminary results on Bessel process.
 \subsection{Preliminaries on Bessel process}
 For $\gamma\in\R$, let $(\rho_t)$ denote a $\gamma$-dimensional Bessel process starting from $r>0$ under $P_{r}^{(\gamma)}$ and let $(\cF_t^\rho)$ denote the right-continuous filtration generated by the Bessel process. Define $\tau_R=\inf\{t\geq 0: \rho_t\leq R\}$ for $R>0$.
The following results (i) and (ii) are from Lemma 5.2 and Lemma 5.3 of \cite{MP17} and the last one follows from (ii) and a simple application of Cauchy-Schwartz inequality.
 \begin{lemma} \label{l13.5} 
Assume $0<2\gamma \le \nu^2$ and $q>2$. Then\\
(i) \[E_r^{(2+2\nu)}\Bigl(\exp\Bigl(\int_0^{\tau_1}\frac{\gamma}{\rho_s^2}\,ds\Bigr)\Bigr|\tau_1<\infty\Bigr)=r^{\nu-\sqrt{\nu^2-2\gamma}}, \forall r\geq 1.\]
(ii)
\[\sup_{r\ge 1}E_r^{(2+2\nu)}\Bigl(\exp\Bigl(\int_0^{\tau_1}\frac{\gamma}{\rho_s^q}\,ds\Bigr)\Bigr|\tau_1<\infty\Bigr)\le C_{\ref{l13.5}}(q,\nu)<\infty.\]
(iii)
\[\inf_{r\ge 1}E_r^{(2+2\nu)}\Bigl(\exp\Bigl(-\int_0^{\tau_1}\frac{\gamma}{\rho_s^q}\,ds\Bigr)\Bigr|\tau_1<\infty\Bigr)\ge
c_{\ref{l13.5}}(q,\nu)>0.\]
\end{lemma}
We also state a result on the application of Girsanov's theorem on Bessel process from \cite{Yor92} (see also Proposition 2.5 of \cite{MP17}).
\begin{lemma}\label{l12.4}
Let $\lambda\geq 0$, $\mu\in \R, r>0$ and $\nu=\sqrt{\lambda^2+\mu^2}$. If $\Phi_t\geq 0$ is $\cF_t^\rho$-adapted, then for all $R<r$, we have
\[E_r^{(2+2\mu)}\Big(\Phi_{t\wedge \tau_R} \exp\Big(-\frac{\lambda^2}{2} \int_0^{t\wedge \tau_R} \frac{1}{\rho_s^2} ds\Big)\Big)=r^{\nu-\mu} E_{r}^{(2+2\nu)} \Big((\rho_{t\wedge \tau_R})^{-\nu+\mu} \Phi_{t\wedge \tau_R}\Big).\]
\end{lemma}
Now define
\begin{align}\label{ev1.5}
\mu=
\begin{cases}
-1/2\ &\text{if}\ d=1,\\
0\ &\text{if}\ d=2,\\
1/2\ &\text{if}\ d=3,
\end{cases}
\text{ and } \nu=\sqrt{\mu^2+4(4-d)}=
\begin{cases}
7/2 &\text{if}\ d=1,\\
2\sqrt{2}\ &\text{if}\ d=2,\\
\frac{\sqrt{17}}{2}\ &\text{if}\ d=3,
\end{cases}
\end{align}
so that 
$d=2+2\mu \text{ and } p=\mu+\nu$ (recall \eqref{ep1.1}).  Let $B$ denote a $d$-dimensional Brownian motion starting from $x$ under $P_x$ for $d\leq 3$  and by slightly abusing the notation, we define $\tau_\eps=\tau_\eps^B=\inf\{t\geq 0: |B_t|\leq \eps\}$ for any $\eps>0$.

 \begin{proposition}\label{p20.1}
Let $x\in \R^d-\{0\}$ and $0<\eps<|x|$. For any Borel measurable function $g: \R^+ \to \R$ bounded on $[r_0, r_0^{-1}]$ for any $r_0>0$, we have
\begin{align*}
E_{x}\Big(&1(\tau_\eps<\infty) \exp\big(-\int_0^{\tau_{\eps}} g(|B_s|) ds\big)\Big)\nn \\
&=\eps^p |x|^{-p}{E}_{|x|}^{(2+2\nu)}\Big( \exp\big(-\int_0^{\tau_{\eps}} (g(\rho_s)-V^\infty(\rho_s)) ds\big)\Big|\tau_\eps<\infty\Big),
\end{align*}
where $\nu$ is as in \eqref{ev1.5}.
\end{proposition}
\begin{proof}
Recall $\mu$ as in \eqref{ev1.5}. Use $d=2+2\mu$  and monotone convergence to see that
\begin{align*}
I:=& E_{x}\Big(1(\tau_\eps<\infty) \exp\big(-\int_0^{\tau_{\eps}} g(|B_s|) ds\big)\Big)=E_{|x|}^{(2+2\mu)}\Big(1(\tau_\eps<\infty) \exp\big(-\int_0^{\tau_{\eps}} g(\rho_s) ds\big)\Big)\nn \\
=&\lim_{t\to \infty} E_{|x|}^{(2+2\mu)}\Big(1(\tau_\eps\leq t)  \exp\big(-\int_0^{\tau_{\eps}} g(\rho_s) ds\big)\Big)\\
=&\lim_{t\to \infty} E_{|x|}^{(2+2\mu)}\Big(1(\tau_\eps\leq \tau_\eps \wedge t)  \exp\big(-\int_0^{\tau_{\eps}\wedge t} g(\rho_s) ds\big)\Big).
\end{align*}
Recall from \eqref{ev1.3} that $V^\infty(x)=\lambda_d |x|^{-2}$ where $\lambda_d=2(4-d)$. Apply Lemma \ref{l12.4} with $\lambda=\sqrt{2\lambda_d}$ and $\mu, \nu$ as in \eqref{ev1.5} to get 
\begin{align*}
I=&\lim_{t\to \infty} E_{|x|}^{(2+2\mu)}\Big(1(\tau_\eps\leq \tau_\eps \wedge t) \exp\big(-\int_0^{\tau_{\eps}\wedge t} (g(\rho_s)-V^\infty(\rho_s)) ds\big) \exp\big(- \int_0^{\tau_{\eps}\wedge t} \frac{\lambda_d}{\rho_s^2} ds\big)\Big)\\
=&\lim_{t\to \infty}|x|^{\nu-\mu} E_{|x|}^{(2+2\nu)}\Big(1(\tau_\eps\leq \tau_\eps \wedge t)  \exp\big(-\int_0^{\tau_{\eps}\wedge t} (g(\rho_s)-V^\infty(\rho_s)) ds\big) \rho_{\tau_{\eps}\wedge t}^{-\nu+\mu}\Big)\\
=&\lim_{t\to \infty}|x|^{\nu-\mu} \eps^{-\nu+\mu} E_{|x|}^{(2+2\nu)}\Big(1(\tau_\eps\leq t)  \exp\big(-\int_0^{\tau_{\eps}} (g(\rho_s)-V^\infty(\rho_s)) ds\big) \Big)\\
=&|x|^{\nu-\mu} \eps^{-\nu+\mu} E_{|x|}^{(2+2\nu)}\Big(1(\tau_\eps<\infty)  \exp\big(-\int_0^{\tau_{\eps}} (g(\rho_s)-V^\infty(\rho_s)) ds\big) \Big)\\
=&|x|^{-\nu-\mu} \eps^{\nu+\mu} E_{|x|}^{(2+2\nu)}\Big( \exp\big(-\int_0^{\tau_{\eps}} (g(\rho_s)-V^\infty(\rho_s)) ds\big)\Big|\tau_\eps<\infty \Big),
\end{align*}
where the second last line uses monotone convergence theorem and the last follows by $P_{|x|}^{(2+2\nu)}(\tau_\eps<\infty)=\eps^{2\nu}|x|^{-2\nu}$.
The proof is complete since we have $p=\mu+\nu$.
 \end{proof}

\subsection{Proof of Theorem \ref{t1}}

By using Proposition \ref{pv0.2}(ii), we have for any $|x|>\eps>0$,
\begin{align}\label{ae2.3}
&\N_0\Big(\frac{X_{G_\eps^x}(1)}{\eps^p} \exp(-\kappa \frac{X_{G_\eps^x}(1)}{\eps^{2}})1(X_{G_{\eps/2}^x}=0) \Big)\nn\\
=&\N_0\Big(\frac{X_{G_\eps^x}(1)}{\eps^p} \exp(-\kappa \frac{X_{G_\eps^x}(1)}{\eps^{2}})\P_{X_{G_\eps^x}}\big(X_{G_{\eps/2}^x}=0\big) \Big)\nn\\
=&\N_0\Big(\frac{X_{G_\eps^x}(1)}{\eps^p} \exp(-\kappa \frac{X_{G_\eps^x}(1)}{\eps^{2}})\exp\big(-U^{\infty, \eps/2}(\eps) X_{G_\eps^x}(1)\big)\Big)\nn\\
=&\N_0\Big(\frac{X_{G_\eps^x}(1)}{\eps^p} \exp\Big(-(\kappa+4U^{\infty, 1}(2)) \frac{X_{G_\eps^x}(1)}{\eps^{2}}\Big)\Big),
\end{align}
where the second equality follows from \eqref{ev5.5} and the last is by the scaling as in \eqref{ev5.3}. Similarly for any $x\in S(X_0)^c$ and $0<\eps<d(x,S(X_0))$, we may use Proposition \ref{pv0.1}(b) to see that
\begin{align}\label{ae2.4}
&\E_{X_0}\Big(\frac{X_{G_\eps^x}(1)}{\eps^p} \exp(-\kappa \frac{X_{G_\eps^x}(1)}{\eps^{2}})1(X_{G_{\eps/2}^x}=0) \Big)\nn\\
=&\E_{X_0}\Big(\frac{X_{G_\eps^x}(1)}{\eps^p} \exp\Big(-(\kappa+4U^{\infty, 1}(2)) \frac{X_{G_\eps^x}(1)}{\eps^{2}}\Big)\Big).
\end{align}
Therefore Theorem \ref{t1} will an easy consequence of the following results.
\begin{proposition}\label{p0.1}
Let $d\leq 3$ and $X_0\in M_F$. For any $\lambda>0$, there is some constant $K_{\ref{p0.1}}(\lambda)>0$ such that
 \begin{align}\label{e9.1.0}
&\lim_{\eps \downarrow 0} \N_0\Big(\frac{X_{G_\eps^x}(1)}{\eps^p} \exp(-\lambda \frac{X_{G_\eps^x}(1)}{\eps^{2}}) \Big)=K_{\ref{p0.1}}(\lambda) |x|^{-p}, \ \forall x\neq 0,
\end{align}
and for any $x\in S(X_0)^c$, we have
 \begin{align}\label{ea9.1.0}
\lim_{\eps \downarrow 0} \E_{X_0}&\Big(\frac{X_{G_\eps^x}(1)}{\eps^p} \exp(-\lambda \frac{X_{G_\eps^x}(1)}{\eps^{2}}) \Big)\nn\\
&= e^{-\int V^\infty(y-x)X_0(dy)}K_{\ref{p0.1}}(\lambda) \int |y-x|^{-p} X_0(dy).
\end{align}
Moreover if $\lambda=\lambda_d$, for any $x\neq 0$ we have 
 \begin{align}\label{e9.4.3}
  \N_0\Big(\frac{X_{G_\eps^x}(1)}{\eps^p} \exp(-\lambda_d \frac{X_{G_\eps^x}(1)}{\eps^{2}}) \Big)=|x|^{-p}, \ \forall 0<\eps<|x|,
\end{align}
and for any $x\in S(X_0)^c$, we have
 \begin{align}\label{ea9.4.3}
\E_{X_0}&\Big(\frac{X_{G_\eps^x}(1)}{\eps^p} \exp(-\lambda_d \frac{X_{G_\eps^x}(1)}{\eps^{2}}) \Big)\nn\\
&= e^{-\int V^\infty(y-x)X_0(dy)} \int |y-x|^{-p} X_0(dy), \ \forall 0<\eps<d(x, S(X_0)).
\end{align}
\end{proposition}
\begin{proof}[Proof of Theorem \ref{t1}]
The proof of \eqref{ce9.3.0} and \eqref{cea9.1.0} is immediate by applying \eqref{ae2.3}, \eqref{ae2.4} and \eqref{e9.1.0}, \eqref{ea9.1.0} with $\lambda=\kappa+4U^{\infty, 1}(2)$ and by letting $C_{\ref{t1}}(\kappa)=K_{\ref{p0.1}}(\kappa+4U^{\infty, 1}(2))$.

 Turning to the upper bounds, by (4.1) of \cite{HMP18} we have $4U^{\infty,1}(2)\geq 4V^\infty(2)=\lambda_d$ and so $\kappa+4U^{\infty, 1}(2)\geq \lambda_d$. Therefore the upper bounds in \eqref{ce9.1.0} and \eqref{cea9.3.0} will follow immediately from \eqref{ae2.3}, \eqref{ae2.4} and \eqref{e9.4.3}, \eqref{ea9.4.3}.
\end{proof}
It remains to prove Proposition \ref{p0.1}. Recall from \eqref{ev5.4} that 
 \begin{align}\label{eb1.0}
 U^{\lambda \eps^{-2},\eps}(x)=\N_0\Big(1-\exp(-\lambda \frac{X_{G_\eps^x}(1)}{\eps^{2}}) \Big), \forall |x|>\eps.
   \end{align}
Monotone convergence and the convexity of $e^{-ax}$ for $a,x>0$ allow us to differentiate the above with respect to $\lambda>0$ through the expectation so that for any $\lambda>0$ we can define
 \begin{align}\label{ae12.2}
U_1^{\lambda \eps^{-2}, \eps}(x):=\frac{d}{d\lambda}U^{\lambda \eps^{-2}, \eps}(x)=\N_{0}\Big(\frac{X_{G_\eps^x}(1)}{\eps^{2}}\exp(-\lambda \frac{X_{G_\eps^x}(1)}{\eps^{2}}) \Big), \forall |x|>\eps.
  \end{align}
 By the Palm measure formula for $X_{G_\eps}$ (see Proposition 4.1 of \cite{Leg94}) applied with\break  $F(y,X_{G_\eps})= \exp(-\lambda \eps^{-2}X_{G_\eps}(1))$, we  can see that for any $x\neq 0$ and $\eps\in (0,|x|)$,
\begin{align}\label{7.17}
\N_x\Big(X_{G_\eps}(1) &\exp(-\lambda \frac{X_{G_\eps}(1)}{\eps^{2}}) \Big)\nn\\
&=\int P_x (dw)1_{(\tau_{\varepsilon}(w)<\infty)} E^{(w)}\Big(\exp(-\lambda \eps^{-2}\int \cN_w (d\kappa )X_{G_\eps}(\kappa)(1) )\Big),
  \end{align}
  where $P_x$ denotes the law of $d$-dimensional Brownian motion starting from $x$ and \break $\tau_\eps(w)=\inf\{t\geq 0: |w(t)|\leq \eps\}$ and for each $w\in \cW_x$ (stopped paths starting from $x$), under $E^{(w)}$, $\cN_w$ denotes a Poisson measure on $C(\R_+, \cW)$ with intensity $\int_0^{\zeta_w} \N_{w(t)}( \cdot) dt$.   Note here we have taken our branching rate for $X$ to be one and so our constants will differ from \cite{Leg94}. The remark below Proposition 4.1 of \cite{Leg94} also implies $\zeta_w=\tau_\eps$. Hence the right-hand side of \eqref{7.17} equals
  \begin{align}\label{e1.00}
  &\int P_x (dw)1_{(\tau_{\varepsilon}(w)<\infty)} \exp\Big(-\int_0^{\tau_\eps} \N_{w(t)}(1-\exp(-\lambda \eps^{-2} X_{G_\eps}(1) ) dt\Big)\nonumber\\
  =&\int P_x (dw)1_{(\tau_{\varepsilon}(w)<\infty)} \exp\Big(-\int_0^{\tau_\eps} U^{\lambda \eps^{-2}, \eps}(w(t))dt \Big)\nn\\
  =&\eps^p |x|^{-p}E_{|x|}^{(2+2\nu)}\Big(  \exp\Big(-\int_0^{ \tau_{{\varepsilon}}} (U^{\lambda \eps^{-2}, \eps}-V^{\infty})(\rho_s)  ds\Big)\Big|\tau_{\varepsilon}<\infty \Big),
   \end{align}
where in the first equality we have used \eqref{eb1.0} and the last follows from Proposition \ref{p20.1} with $g=U^{\lambda \eps^{-2}, \eps}$. Combining \eqref{ae12.2}, \eqref{7.17} and \eqref{e1.00}, we conclude
   \begin{align}\label{ae1.0}
\frac{1}{\eps^{p-2}}U_1^{\lambda \eps^{-2}, \eps}(x)=&\N_0\Big(\frac{X_{G_\eps^x}(1)}{\eps^p} \exp(-\lambda \frac{X_{G_\eps^x}(1)}{\eps^{2}}) \Big)=\N_x\Big(\frac{X_{G_\eps}(1)}{\eps^p} \exp(-\lambda \frac{X_{G_\eps}(1)}{\eps^{2}}) \Big)\nn\\
=& |x|^{-p}E_{|x|}^{(2+2\nu)}\Big(  \exp\Big(-\int_0^{ \tau_{{\varepsilon}}} (U^{\lambda \eps^{-2}, \eps}-V^{\infty})(\rho_s)  ds\Big)\Big|\tau_{\varepsilon}<\infty \Big),
  \end{align}
  where the second equality is by translation invariance and spherical symmetry.
Recall $\lambda_d=2(4-d)$ as in \eqref{ev1.3}. As one can easily check, $V^\infty(x)=\lambda_d |x|^{-2}$ is also a solution to the PDE \[\Delta U^{\lambda_d,1}=(U^{\lambda_d,1})^2 \text{ on } G_1, \quad U^{\lambda_d,1}= \lambda_d \text{ on } \partial G_1,\] and thanks to the uniqueness we have 
   \begin{align}\label{e1.2}
  U^{\lambda_d,1}(x)=V^\infty(x), \forall |x|\geq 1.
  \end{align}
  By the scaling of $U^{\lambda\eps^{-2}, \eps}$ from \eqref{ev5.3}, one can see that \eqref{e1.2} implies
\begin{align}\label{ev5.7}
U^{\lambda_d\eps^{-2}, \eps}(x)=\eps^{-2} U^{\lambda_d,1}(x/\eps) = \eps^{-2}V^{\infty}(x/\eps)=V^{\infty}(x), \forall |x|\geq \eps.
\end{align}
Therefore by letting $\lambda=\lambda_d$ in \eqref{ae1.0}, we can get 
   \begin{align}\label{e1.1}
\frac{1}{\eps^{p-2}} U_1^{\lambda_d\eps^{-2}, \eps}(x)=\N_0\Big(\frac{X_{G_\eps^x}(1)}{\eps^p} \exp(-\lambda_d \frac{X_{G_\eps^x}(1)}{\eps^{2}}) \Big)=|x|^{-p}, \forall |x|>\eps.
  \end{align}
If $\lambda>\lambda_d$, by \eqref{e1.2} and the definition of $U^{\lambda,1}$ as in \eqref{eb1.0}, we have 
\begin{align}\label{ae1.8}
U^{\lambda,1}(x)\geq U^{\lambda_d,1}(x)= V^\infty(x) \text{ for all } |x|\geq 1.
\end{align}
Use scaling as in \eqref{ev5.7} to conclude 
\begin{align}\label{ev5.8}
U^{\lambda \eps^{-2}, \eps}(x)\geq V^{\infty}(x), \forall |x|\geq \eps \text{ if } \lambda>\lambda_d.
\end{align}
Similarly if $0<\lambda<\lambda_d$, we have 
\begin{align}\label{ae1.9}
U^{\lambda,1}(x)\leq  V^\infty(x), \forall |x|\geq 1,
\end{align}
and 
\begin{align}\label{ae1.90}
U^{\lambda\eps^{-2}, \eps}(x)\leq V^{\infty}(x), \forall |x|\geq \eps.
\end{align}


\begin{lemma} \label{ac13.5} 
For any $\lambda>0$, we have for all $x\neq 0$,
\begin{align}\label{ae13.5}
\lim_{\eps\downarrow 0} E_{|x|}^{(2+2\nu)}\Big(\exp\big(-\int_0^{ \tau_{\eps}}  (U^{\lambda \eps^{-2}, \eps}-V^\infty)(\rho_s) ds\big)\Big|\tau_{\eps}<\infty\Big)=K_{\ref{p0.1}}(\lambda)\in (0,\infty).
\end{align}
Moreover, if $0<\lambda<\lambda_d$, we have
\begin{align}\label{ce1.1}
\sup_{0<\eps<|x|} E_{|x|}^{(2+2\nu)}\Big(\exp\big(-\int_0^{ \tau_{\eps}}  (U^{\lambda \eps^{-2}, \eps}-V^\infty)(\rho_s) ds\big)\Big|\tau_{\eps}<\infty\Big)=K_{\ref{p0.1}}(\lambda)<\infty.
\end{align}
\end{lemma}

\begin{proof}
For  any $\lambda>0$ and any $r\geq 1$, we define 
  \begin{align}\label{ae3.1}
  f^\lambda(r)&:=E_{r}^{(2+2\nu)}\Big(\exp\big(-\int_0^{ \tau_{1}}  (U^{\lambda, 1}-V^\infty)(\rho_s) ds\big)\Big|\tau_{1}<\infty\Big).
\end{align}  
By Lemma \ref{l13.5}(i) and the definition of $\nu$ as in \eqref{ev1.5}, for any $\lambda>0$ we have 
\[f^\lambda(r)\leq  E_{r}^{(2+2\nu)}\Big(\exp\big(\int_0^{ \tau_{1}} V^\infty(\rho_s) ds\big)\Big|\tau_{1}<\infty\Big)=r^{\nu-|\mu|}<\infty, \forall r\geq 1. \]
For any $0<\eps<|x|$, we use the scaling of Bessel process and the scaling of $U^{\lambda\eps^{-2}, \eps}$ as in \eqref{ev5.7} to get
  \begin{align}\label{ae7.3}
  &E_{|x|}^{(2+2\nu)}\Big(\exp\big(-\int_0^{ \tau_{\eps}}  (U^{\lambda \eps^{-2}, \eps}-V^\infty)(\rho_s) ds\big)\Big|\tau_{\eps}<\infty\Big)\nn\\
  =&E_{|x|/\eps}^{(2+2\nu)}\Big(\exp\big(-\int_0^{ \tau_{1}}  (U^{\lambda, 1}-V^\infty)(\rho_s) ds\big)\Big|\tau_{1}<\infty\Big)=f^\lambda(|x|/\eps).
\end{align}
Then it suffices to find some finite constant $K_{\ref{p0.1}}(\lambda)>0$ such that $\lim_{r\to \infty} f^\lambda(r)=K_{\ref{p0.1}}(\lambda)$ to finish the proof of \eqref{ae13.5}.\\
 
  For any $r>R>1$, we have
\begin{align}\label{ae7.1}
f^\lambda(r)=&E_{r}^{(2+2\nu)}\Big(\exp\big(-\int_0^{ \tau_{1}}  (U^{\lambda, 1}-V^\infty)(\rho_s) ds\big)1(\tau_{1}<\infty)\Big) r^{2\nu}\nn\\
=&E_{r}^{(2+2\nu)} \Big( \exp\Big(-\int_0^{\tau_R}  (U^{\lambda,1}-V^\infty)(\rho_s)ds\Big)1(\tau_R<\infty)\Big) \nn\\
&\quad \quad \quad \quad \quad E_{R}^{(2+2\nu)} \Big( \exp\Big(-\int_0^{\tau_1}  (U^{\lambda,1}-V^\infty)(\rho_s) ds\Big) 1(\tau_1<\infty)\Big) r^{2\nu}\nn\\
= &E_{r}^{(2+2\nu)} \Big( \exp\Big(-\int_0^{\tau_R}  (U^{\lambda,1}-V^\infty)(\rho_s)ds\Big)\Big| \tau_R<\infty\Big) \nn\\
&\quad \quad \quad \quad \quad E_{R}^{(2+2\nu)} \Big( \exp\Big(-\int_0^{\tau_1}  (U^{\lambda,1}-V^\infty)(\rho_s) ds\Big)\Big|\tau_1<\infty\Big) \nn\\
= &E_{r}^{(2+2\nu)} \Big( \exp\Big(-\int_0^{\tau_R}  (U^{\lambda,1}-V^\infty)(\rho_s)ds\Big)\Big| \tau_R<\infty\Big)f^\lambda(R),
\end{align}
where the first and the third equalities follow by $P_{r}^{(2+2\nu)}(\tau_R<\infty)=(R/r)^{2\nu}$ for any $r>R>0$ and in the second equality we have used the strong Markov Property of Bessel process.
By \eqref{ev1.4} and \eqref{ev5.2}, for any $\lambda>0$ we obtain
 \[\frac{\Delta}{2} (U^{\lambda,1}-V^\infty)(x)=\frac{(U^{\lambda,1}+V^\infty)(x)}{2} (U^{\lambda,1}-V^\infty)(x), \forall |x|>1.\] Then the Feyman-Kac formula (as in (3.8) in \cite{MP17}) will give us that for  $|x|\geq R>1$,
\begin{align}\label{e1.0}
(U^{\lambda,1}-V^\infty)(x)=&(U^{\lambda,1}-V^\infty)(R)\nn\\
&E_x\Big(1_{(\tau_R<\infty)} \exp\Bigl(-\int_{0}^{\tau_R} \Bigl(\frac{U^{\lambda,1}+V^{\infty}}{2}\Bigr)(B_s)ds\Bigr)\Big),
\end{align}
where $B$ denotes a $d$-dimensional Brownian motion starting at $x$ under $P_x$ and \break $\tau_R=\inf \{t\geq 0: |B_t|\leq R\}$.\\

\no (i) If $\lambda=\lambda_d$, by \eqref{e1.2} and \eqref{ae3.1} we have $f^{\lambda_d}(r)\equiv 1$ for any $r>1$ and the result follows immediately by letting $K_{\ref{p0.1}}(\lambda_d)=1$.\\

\no (ii) If $\lambda>\lambda_d$, by \eqref{ae1.8} and \eqref{ae3.1}  we have $f^\lambda(r)\leq 1$ for all $r>1$. 
Apply \eqref{ae1.8} again in \eqref{ae7.1} to see that $f^{\lambda}(r)\leq f^{\lambda}(R)$ for any $r>R>1$ and so conclude 
%
$r\mapsto f^\lambda(r)$ is monotone decreasing for $r>1$. If we can show that $\inf_{r>1} f^\lambda(r)>0$, then we can find some constant  $K_{\ref{p0.1}}(\lambda)>0$ so that $$\inf_{r> 1} f^\lambda(r)=\lim_{r\to \infty} f^\lambda(r)=K_{\ref{p0.1}}(\lambda),$$ and the proof is complete. To see this, for any $|x|\geq R>1$, we apply \eqref{ae1.8} in \eqref{e1.0} to see that
\begin{align*}
(U^{\lambda,1}-V^\infty)(x)&\leq (U^{\lambda,1}-V^\infty)(R)E_x\Big(1_{(\tau_R<\infty)} \exp\Bigl(-\int_{0}^{\tau_R} V^{\infty}(B_s)ds\Bigr)\Big)\\
&=(U^{\lambda,1}-V^\infty)(R) R^p|x|^{-p},
\end{align*}
where the last equality uses Proposition \ref{p20.1} with $g=V^\infty$.
Let $R\downarrow 1$ to get 
\begin{align} \label{e1.6}
    (U^{\lambda,1}-V^\infty)(x)\leq (U^{\lambda,1}-V^\infty)(1)|x|^{-p}=(\lambda- \lambda_d) |x|^{-p}, \quad \forall |x|\geq 1.
\end{align} 
Apply the above bound to see that for all $r>R>1$,
\begin{align*}
&E_{r}^{(2+2\nu)} \Big( \exp\Big(-\int_0^{\tau_R}  (U^{\lambda,1}-V^\infty)(\rho_s)ds\Big)\Big| \tau_R<\infty\Big)\\
 \geq& E_{r}^{(2+2\nu)} \Big( \exp\Big(-\int_0^{\tau_R}  (\lambda-\lambda_d) \rho_s^{-p} ds\Big)\Big| \tau_R<\infty\Big) \\
=& E_{r/R}^{(2+2\nu)} \Big( \exp\Big(-\int_0^{\tau_1}   (\lambda-\lambda_d) R^{2-p} \rho_s^{-p}ds\Big)\Big| \tau_1<\infty\Big) \geq c_{\ref{l13.5}}(p,\nu),
\end{align*}
where in the last inequality we have chosen $R>1$ large so that $2(\lambda-\lambda_d) R^{2-p}<\nu^2$ and then applied Lemma \ref{l13.5}(iii). The equality is by scaling of Bessel process. Recalling \eqref{ae7.1}, we choose $R>1$ large and then use the above to get for all $r>R$,
\begin{align*}
f^\lambda(r)= &E_{r}^{(2+2\nu)} \Big( \exp\Big(-\int_0^{\tau_R}  (U^{\lambda,1}-V^\infty)(\rho_s)ds\Big)\Big| \tau_R<\infty\Big)f^\lambda(R) \geq  c_{\ref{l13.5}}(p,\nu) f^\lambda(R)>0.
 \end{align*}
Hence it follows that $\inf_{r>1} f^\lambda(r)>0$ and the proof is complete.\\

\no (iii) If $\lambda \in (0,\lambda_d)$, by \eqref{ae1.9} and \eqref{ae7.1} one can easily check that $f^{\lambda}(r)\geq f^\lambda(R)$ for any $r>R>1$ and so conclude $r\mapsto f^\lambda(r)$ is monotone increasing for $r>1$. If we show that $\sup_{r>1} f^\lambda(r)<\infty$, then we can find some constant  $K_{\ref{p0.1}}(\lambda)>0$ so that 
 \begin{align}\label{e20.1}
 \sup_{r> 1} f^\lambda(r)=\lim_{r\to \infty} f^\lambda(r)=K_{\ref{p0.1}}(\lambda).
 \end{align}
Moreover, the proof of \eqref{ce1.1} will also follow from \eqref{ae3.1}, \eqref{ae7.3} and \eqref{e20.1}.

It remains to prove $\sup_{r>1} f^\lambda(r)<\infty$. We first consider $0<\lambda<1$. By \eqref{e1.0} we have for all $|x|>R>1$,
 \begin{align}\label{ef1.9}
(V^\infty-U^{\lambda,1})(x)=(V^\infty-U^{\lambda,1})(R)E_x\Big(1_{(\tau_R<\infty)} \exp\Bigl(-\int_{0}^{\tau_R} \Bigl(\frac{U^{\lambda,1}+V^{\infty}}{2}\Bigr)(B_s)ds\Bigr)\Big).
\end{align}
Recall from (4.13) of \cite{HMP18} that for any $\lambda \in (0,1)$,
\begin{align}\label{ef1.1}
\forall \delta\in (0,1),\ \exists C_\delta>2,  \text{ so that } U^{\lambda,1}(x)\geq (1-\delta) V^{\infty}(x), \forall  |x|\geq C_\delta/\lambda. 
\end{align}
Then we can apply the proof of Proposition 4.6(b) in \cite{HMP18} with \eqref{ef1.1} and \eqref{ef1.9} to see that there exist some constants $C, K>0$ such that (cf. (4.14) of \cite{HMP18})
\begin{align}\label{e1.3}
(V^\infty-U^{\lambda,1})(x)\leq C (V^\infty-U^{\lambda,1})(R) (R/|x|)^p, \quad \forall |x|\geq R\geq \frac{K}{\lambda}>1.
\end{align}
Let $r>R>R_0=K/\lambda$, where $R$ will be chosen to be some large constant below. Now let  $\xi(R_0)=C(V^\infty-U^{\lambda,1})(R_0) R_0^p$ and apply the above bound to get for all $r>R>R_0>1$,
\begin{align}\label{ef1.2}
&E_{r}^{(2+2\nu)} \Big( \exp\Big(\int_0^{\tau_R}  (V^\infty-U^{\lambda,1})(\rho_s)ds\Big)\Big| \tau_R<\infty\Big)\nn\\
 \leq& E_{r}^{(2+2\nu)} \Big( \exp\Big(\int_0^{\tau_R}   \xi(R_0) \rho_s^{-p} ds\Big)\Big| \tau_R<\infty\Big)\nn \\
=& E_{r/R}^{(2+2\nu)} \Big( \exp\Big(\int_0^{\tau_1}   \xi(R_0) R^{2-p} \rho_s^{-p}ds\Big)\Big| \tau_1<\infty\Big) \leq C_{\ref{l13.5}}(p,\nu),
\end{align}
where in the last inequality we have chosen $R>1$ large so that $2\xi(R_0) R^{2-p}<\nu^2$ and then applied Lemma \ref{l13.5}(ii). Recalling \eqref{ae7.1}, by choosing $R>1$ large, we have for all $r>R$,
\begin{align*}
f^\lambda(r)= &E_{r}^{(2+2\nu)} \Big( \exp\Big(-\int_0^{\tau_R}  (U^{\lambda,1}-V^\infty)(\rho_s)ds\Big)\Big| \tau_R<\infty\Big)f^\lambda(R) \\
\leq &  C_{\ref{l13.5}}(p,\nu) f^\lambda(R)<\infty,
 \end{align*}
The last follows since
 \begin{align*}
 f^\lambda(R)\leq E_{R}^{(2+2\nu)} \Big( \exp\Big(\int_0^{\tau_1} V^\infty(\rho_s) ds\Big)\Big|\tau_1<\infty\Big)=R^{\nu-\sqrt{\nu^2-4(4-d)}}<\infty,
 \end{align*}
 where we have used  Lemma \ref{l13.5}(i) in the equality. Hence it follows that $\sup_{r>1} f^\lambda(r)<\infty$ for any $0<\lambda<1$.
 
  If $1\leq \lambda<\lambda_d$, we have $\lambda\geq 1/2$ and so by definition of $U^{\lambda,1}$ as in \eqref{eb1.0}, we have $U^{\lambda,1}(x)\geq U^{1/2,1}(x)$ for all $|x|>1$. It follows from the definition of $ f^\lambda(r)$(recall \eqref{ae3.1}) that
 \[\sup_{r> 1} f^\lambda(r)\leq \sup_{r> 1} f^{1/2}(r)<\infty.\]
Now we conclude that for any $0<\lambda<\lambda_d$, we have $\sup_{r>1} f^\lambda(r)<\infty$ and so the proof is complete.
 \end{proof}
 We are ready to finish the Proof of Proposition \ref{p0.1}.
\begin{proof}[Proof of Proposition \ref{p0.1}]
The proof of \eqref{e9.1.0} is immediate by \eqref{ae1.0} and Lemma \ref{ac13.5}. Turning to $\P_{X_0}$, for any $x\in S(X_0)^c$, we let $\delta:=d(x, S(X_0))>0$ and $0<\eps<\delta/2$.
Recall from \eqref{ev5.4} that 
\begin{align}\label{e1.2.4}
\E_{X_0}\Big(\exp(-\lambda \frac{X_{G_\eps^x}(1)}{\eps^{2}}) \Big)=  \exp\Big(-\int U^{\lambda \eps^{-2}, \eps}(x-y) X_0(dy)\Big).
\end{align}
Recalling \eqref{ae12.2} and \eqref{e1.1}, we may conclude
   \begin{align}\label{ea1.3}
\frac{1}{\eps^{p-2}} U_1^{\lambda\eps^{-2}, \eps}(x)\leq \frac{1}{\eps^{p-2}} U_1^{\lambda_d\eps^{-2}, \eps}(x)= |x|^{-p}\text{ for all } |x|>\eps,  \text{ if } \lambda\geq \lambda_d.
  \end{align}
If $ 0<\lambda<\lambda_d,$ then \eqref{ae1.0} and \eqref{ce1.1} imply that
   \begin{align}\label{ea1.4}
\frac{1}{\eps^{p-2}} U_1^{\lambda\eps^{-2}, \eps}(x)\leq K_{\ref{p0.1}}(\lambda) |x|^{-p}, \forall |x|>\eps.
  \end{align}
Monotone convergence and the convexity of $e^{-ax}$ for $a,x>0$ allow us to differentiate \eqref{e1.2.4} with respect to $\lambda>0$ through the expectation such that for any $\lambda>0$,
\begin{align}\label{e1.2.7}
\E_{X_0}\Big(\frac{X_{G_\eps^x}(1)}{\eps^2}&\exp(-\lambda \frac{X_{G_\eps^x}(1)}{\eps^{2}}) \Big)\nn\\
&=   \int U_1^{\lambda \eps^{-2}, \eps}(y-x) X_0(dy)\exp\Big(-\int U^{\lambda \eps^{-2}, \eps}(y-x) X_0(dy)\Big).
\end{align}
The differentiation through the integral with respect to $X_0$ on the right-hand side of \eqref{e1.2.4} follows since $|y-x|\geq \delta$ for all $y\in S(X_0)$ and by \eqref{ea1.3} and \eqref{ea1.4}, for any $\lambda>0$ we have $\eps^{-(p-2)}U_1^{\lambda \eps^{-2}, \eps}(x)$ is uniformly bounded for all $|x|\geq\delta$. The same reasoning allows us to take limit inside the integral and use \eqref{e9.1.0} and \eqref{ae12.2} to get
\begin{align}\label{e1.2.8}
\lim_{\eps\downarrow 0}\int \frac{1}{\eps^{p-2}}U_1^{\lambda \eps^{-2}, \eps}(y-x) X_0(dy)=\int K_{\ref{p0.1}}(\lambda) |y-x|^{-p} X_0(dy).
\end{align}
 Next by \eqref{e1.2}, \eqref{ae1.8}, \eqref{ae1.9} and \eqref{e1.6}, \eqref{e1.3}, one may easily conclude that (see also (3.10) of \cite{MP17}) 
\begin{align}
\lim_{|x|\to \infty} \frac{U^{\lambda,1}(x)}{V^\infty(x)}=1, \forall \lambda>0.
\end{align}
The scaling of $U^{\lambda \eps^{-2}, \eps}(x)$ as in \eqref{ev5.7} then gives for any $\lambda>0$, 
\begin{align}
 \lim_{\eps \downarrow 0} U^{\lambda \eps^{-2},\eps}(x)=\lim_{\eps \downarrow 0} \eps^{-2}U^{\lambda ,1}(x/\eps)=V^\infty(x), \forall x\neq 0.
\end{align}
Hence it follows that
\begin{align}\label{e1.2.9}
\lim_{\eps \downarrow 0} \int U^{\lambda \eps^{-2}, \eps}(y-x) X_0(dy)=\int V^{\infty}(y-x) X_0(dy),
\end{align}
and the proof of \eqref{ea9.1.0} is complete by \eqref{e1.2.7}, \eqref{e1.2.8} and \eqref{e1.2.9}.\\

\no If $\lambda=\lambda_d$, then \eqref{e9.4.3} will follow immediately from \eqref{e1.1}. The proof of \eqref{ea9.4.3} is also immediate from \eqref{e1.2.7}, \eqref{ev5.7} and \eqref{e1.1}.
\end{proof}
The following is an easy application of Proposition \ref{p0.1} and could be of independent interests. 
\begin{corollary}
For any $X_0\in M_F$ and any $x\in S(X_0)^{c}$, the family \[\Big\{\frac{X_{G_{r_0-r}^x}(1)}{(r_0-r)^p} \exp(-\lambda_d \frac{X_{G_{r_0-r}^x}(1)}{(r_0-r)^{2}})\Big\} \text{ indexed by } 0\leq r<r_0=d(x,S(X_0))/2\] is a martingale under $\N_{X_0}$ or $\P_{X_0}$.
\end{corollary}
\begin{proof}
%
For any $X_0\in M_F$, we fix any $x\in S(X_0)^c$. Let $r_0=d(x,S(X_0))/2$. For any $0<\eps_2<\eps_1<r_0$, we apply Proposition \ref{pv0.2}(ii) with $G_1=G_{\eps_1}^x$ and $G_2=G_{\eps_2}^x$ to get 
 \begin{align}\label{e4.2}
 &\N_{X_0}\Big( \frac{X_{G_{\eps_2}^x}(1)}{\eps_2^p}  \exp(-\lambda_d \frac{X_{G_{\eps_2}^x}(1)}{{\eps_2}^{2}})\Big| \cE_{G_{\eps_1}^x}  \Big)= \E_{X_{G_{\eps_1}^x}} \Big(\frac{X_{G_{\eps_2}^x}(1)}{\eps_2^p}  \exp(-\lambda_d \frac{X_{G_{\eps_2}^x}(1)}{{\eps_2}^{2}})  \Big)\nn\\
    =& \int |y-x|^{-p} X_{G_{\eps_1}^x}(dy) \exp\Big(-\int V^{\infty}(y-x)X_{G_{\eps_1}^x}(dy)\Big)\nn\\
    =&  \frac{X_{G_{\eps_1}^x}(1)}{\eps_1^{p}} \exp\Big(- \lambda_d \frac{X_{G_{\eps_1}^x}(1)}{\eps_1^2}\Big),
   \end{align}
   where the second line is by \eqref{ea9.4.3}.
 Therefore we conclude \[\Big\{\frac{X_{G_{r_0-r}^x}(1)}{(r_0-r)^p} \exp(-\lambda_d \frac{X_{G_{r_0-r}^x}(1)}{(r_0-r)^{2}}), 0\leq r<r_0\Big\}\] is a martingale under $\N_{X_0}$. The case for $\P_{X_0}$ follows in a similar way.
\end{proof}

\bibliographystyle{plain}
\def\cprime{$'$}

 \end{document}